\documentclass[11pt]{amsart}
\usepackage{amssymb,amsmath,epsfig,mathrsfs, enumerate, xparse, mathtools}
\usepackage[pagewise]{lineno}
\usepackage[nodisplayskipstretch]{setspace}
\usepackage{graphicx}\usepackage[normalem]{ulem}
\usepackage{fancyhdr}
\usepackage[margin=2.5cm]{geometry}
\usepackage[pagebackref,colorlinks=true,linkcolor=blue,citecolor=blue]{hyperref}
\hypersetup{
 colorlinks   = true,
 urlcolor     = blue,
 linkcolor    = blue,
 citecolor   = red ,
 bookmarksopen=true
}

\theoremstyle{plain}
\newtheorem{thrm}{Theorem}[section]
\newtheorem{lemma}[thrm]{Lemma}

\newtheorem{rmrk}[thrm]{Remark}

\allowdisplaybreaks
\begin{document}
% begin top matter
% ***************** macroes needed for this paper ************************
\newcommand{\sn}{\mathbb{S}^{n-1}}
\newcommand{\SL}{\mathcal L^{1,p}( D)}
\newcommand{\Lp}{L^p( Dega)}
\newcommand{\py}{  \partial_z^a}
\newcommand{\La}{\mathscr{L}_a}
\newcommand{\CO}{C^\infty_0( \Omega)}
\newcommand{\Rn}{\mathbb R^n}
\newcommand{\Rm}{\mathbb R^m}
\newcommand{\R}{\mathbb R}
\newcommand{\Om}{\Omega}
\newcommand{\Hn}{\mathbb H^n}
\newcommand{\aB}{\alpha B}
\newcommand{\eps}{\ve}
\newcommand{\BVX}{BV_X(\Omega)}
\newcommand{\p}{\partial}
\newcommand{\IO}{\int_\Omega}
\newcommand{\bG}{\boldsymbol{G}}
\newcommand{\bg}{\mathfrak g}
\newcommand{\bz}{\mathfrak z}
\newcommand{\bv}{\mathfrak v}
\newcommand{\Bux}{\mbox{Box}}
\newcommand{\e}{\ve}
\newcommand{\X}{\mathcal X}
\newcommand{\Y}{\mathcal Y}
\newcommand{\W}{\mathcal W}
\newcommand{\la}{\lambda}
\newcommand{\vf}{\varphi}
\newcommand{\rhh}{|\nabla_H \rho|}
\newcommand{\Ba}{\mathcal{B}_\beta}
\newcommand{\Za}{Z_\beta}
\newcommand{\ra}{\rho_\beta}
\newcommand{\n}{\nabla}
\newcommand{\vt}{\vartheta}
\newcommand{\its}{\int_{\{y=0\}}}

\numberwithin{equation}{section}

\newcommand{\RN} {\mathbb{R}^N}
\newcommand{\Sob}{S^{1,p}(\Omega)}
\newcommand{\Dxk}{\frac{\partial}{\partial x_k}}
\newcommand{\Co}{C^\infty_0(\Omega)}
\newcommand{\Je}{J_\ve}
\newcommand{\beq}{\begin{equation}}
\newcommand{\bea}[1]{\begin{array}{#1} }
\newcommand{\eeq}{ \end{equation}}
\newcommand{\ea}{ \end{array}}
\newcommand{\eh}{\ve h}
\newcommand{\Dxi}{\frac{\partial}{\partial x_{i}}}
\newcommand{\Dyi}{\frac{\partial}{\partial y_{i}}}
\newcommand{\Dt}{\frac{\partial}{\partial t}}
\newcommand{\aBa}{(\alpha+1)B}
\newcommand{\GF}{\psi^{1+\frac{1}{2\alpha}}}
\newcommand{\GS}{\psi^{\frac12}}
\newcommand{\HFF}{\frac{\psi}{\rho}}
\newcommand{\HSS}{\frac{\psi}{\rho}}
\newcommand{\HFS}{\rho\psi^{\frac12-\frac{1}{2\alpha}}}
\newcommand{\HSF}{\frac{\psi^{\frac32+\frac{1}{2\alpha}}}{\rho}}
\newcommand{\AF}{\rho}
\newcommand{\AR}{\rho{\psi}^{\frac{1}{2}+\frac{1}{2\alpha}}}
\newcommand{\PF}{\alpha\frac{\psi}{|x|}}
\newcommand{\PS}{\alpha\frac{\psi}{\rho}}
\newcommand{\ds}{\displaystyle}
\newcommand{\Zt}{{\mathcal Z}^{t}}
\newcommand{\XPSI}{2\alpha\psi \begin{pmatrix} \frac{x}{\left< x \right>^2}\\ 0 \end{pmatrix} - 2\alpha\frac{{\psi}^2}{\rho^2}\begin{pmatrix} x \\ (\alpha +1)|x|^{-\alpha}y \end{pmatrix}}
\newcommand{\Z}{ \begin{pmatrix} x \\ (\alpha + 1)|x|^{-\alpha}y \end{pmatrix} }
\newcommand{\ZZ}{ \begin{pmatrix} xx^{t} & (\alpha + 1)|x|^{-\alpha}x y^{t}\\
     (\alpha + 1)|x|^{-\alpha}x^{t} y &   (\alpha + 1)^2  |x|^{-2\alpha}yy^{t}\end{pmatrix}}
\newcommand{\norm}[1]{\lVert#1 \rVert}
\newcommand{\ve}{\varepsilon}
\newcommand{\D}{\operatorname{div}}
\newcommand{\G}{\mathscr{G}}
\newcommand{\sa}{\langle}
\newcommand{\da}{\rangle}

\title[ucp in measure etc]{ On unique continuation in measure for fractional heat equations}

\author{Agnid Banerjee}
\address{School of Mathematical and Statistical Sciences\\ Arizona State University}\email[Agnid Banerjee]{agnidban@gmail.com}

\author{Nicola Garofalo}
\address{School of Mathematical and Statistical Sciences\\ Arizona State University}\email[Nicola Garofalo]{rembrandt54@gmail.com}

%
% 
% AMS information
%

\maketitle

\begin{abstract}
We prove a theorem of unique continuation in measure for nonlocal equations of the type  $(\partial_t - \Delta)^s u= V(x,t) u$,  for $0<s <1$.  Our main result, Theorem \ref{main}, establishes a delicate nonlocal counterpart of the unique continuation in measure for the local case $s=1$.
\end{abstract}

\section{Introduction and statement of main result}

The problem of unique continuation occupies a central position in the analysis of partial differential equations. A fundamental question in the subject is whether the trivial solution is the only one that can vanish to infinite order at one point. When this is the case, one says that the strong unique continuation property holds. In the study of 
observability inequalities and/or null-controllability of parabolic evolutions over measurable sets, a different type of unique continuation becomes relevant: can a nontrivial solution vanish on a subset of positive measure? If this cannot happen, one says that the relevant differential operator has the unique continuation property in measure, see for instance \cite{EMZ} and the references therein.

The primary objective of this paper is to establish a theorem of unique continuation in measure for solutions to the nonlocal parabolic equation 
\begin{equation}\label{e0}
H^s u(x,t) + V(x,t) u(x,t) = 0,\ \ \ \ \ \ \ \ 0<s<1,
\end{equation}
in the space-time cylinder $B_1 \times (-1, 0]\subset \Rn_x\times \R_t$. In \eqref{e0} we have denoted by $H^s = (\p_t - \Delta_x)^s$ the fractional power of the heat operator $H = \p_t - \Delta_x$ in $\R^{n+1} = \Rn_x \times \R_t$, and 
on the potential $V$ suitable assumptions will be specified in \eqref{vassump} below. Throughout this paper, for a function $f:\R^{n+1}\to \mathbb C$, we indicate with 
\[
\hat f(\xi,\sigma) = \int_{\R^{n+1}} e^{-2\pi i (\sa\xi,x\da + \sigma t)} f(x,t) dx dt
\]
its Fourier transform. Then the action of $H^s$ on a function $f\in \mathscr S(\R^{n+1})$ is defined by the formula 
\begin{equation}\label{sHft}
\widehat{H^s f}(\xi,\sigma) = (4\pi^2 |\xi|^2 + 2\pi i \sigma)^s\  \hat f(\xi,\sigma),
\end{equation}
with the understanding that we have chosen the principal branch of the complex function $z\to z^s$. As it is well-known, the nonlocal operator $H^s$ is alternatively given by the formula
\[
H^{s} f(x,t) = - \frac{s}{\Gamma(1-s)} \int_0^\infty \frac{1}{\tau^{1+s}} \big(P^H_\tau f(x,t) - f(x,t)\big) d\tau,
\]
where we have let
\[
P^H_\tau f(x,t) = (4\pi \tau)^{-\frac n2} \int_{\Rn} e^{-\frac{|x-y|^2}{4\tau}} f(y,t-\tau) dy.
\]
The domain of the operator $H^s$ will be denoted by  
\[
\operatorname{Dom}(H^s) = \{u\in L^2(\Rn \times \R)\mid H^s u \in L^2(\Rn \times \R)\}.
\]
We say that a function $u\in \operatorname{Dom}(H^s)$ solves \eqref{e0} in an open set $\Om\subset \R^{n+1}$ if the equation is satisfied for a.e. point $(x,t)\in \Om$. For a measurable set  $E \subset \Rn$, we will indicate by $|E|$ its $n$-dimensional Lebesgue measure. The ball in $\Rn$ centered at the origin and having radius $r>0$ will be denoted by $B_r$. Our main result is the following.

\begin{thrm}\label{main}
Let $u \in \operatorname{Dom}(H^s)$ solve \eqref{e0} in the cylinder $B_1 \times (-1, 0]$.  Assume that $u$ vanishes  in $F \times (-1, 0]$, for some measurable set $F \subset B_1$ with $|F|>0$.  Then $u \equiv 0$ in $\Rn \times (-1,0]$.  
\end{thrm}

The proof of Theorem \ref{main} will be presented in Section \ref{s:m}. We mention that the main new difficulty with this result is the treatment of the nonlocal operator $H^s$. In the local case, in fact, by an application of the Poincar\'e and energy inequalities, it is easy to show that at any Lebesgue point of its zero set, a solution of the relevant differential operator must vanish to infinite order. This reduces the proof to having the strong unique continuation property for the local operator in question, see for instance \cite{DG, EMZ, Ru2, BZ}. In the framework of the present paper, a corresponding space-like strong unique continuation result has been recently established in \cite{ABDG}, see Theorem \ref{main1} below. However, since the energy inequality for fractional equations has a nonlocal character, a straightforward adaptation of the above mentioned local arguments is all but  obvious. This aspect was already mentioned in the introduction of \cite{FF}, where the authors derived the time-independent counterpart of Theorem \ref{main} by analysing the local asymptotic of solutions to the corresponding Caffarelli-Silvestre extension problem for $(-\Delta)^s$.

In this note, we present a different approach which, in fact, does reduce the property of unique continuation in measure to that of strong unique continuation. The novel aspects of our proof consist in exploiting in a subtle way some fundamental properties of the solution of the extension problem associated with \eqref{e0}. One of them is Theorem \ref{extmain}, which we use to establish the crucial compactness Lemma \ref{L:pos}. Another important ingredient is  
the conditional doubling property in Theorem \ref{doub}, which we use in multiple ways in the proof of Theorem \ref{main}. Combining these tools with the trace interpolation inequality in Lemma \ref{L:inter1} (and the Poincar\'e inequality in Lemma \ref{po}), we are able to prove that, if $u$ in Theorem \ref{main} does not vanish identically in the cylinder $B_1 \times (-1, 0]$, then at any Lebesgue point $x_0$ of the set $F$, the following inequality holds for any $\ve>0$ and for all $0<r<r_\ve$ 
\begin{equation*}
\int_{Q_r(x_0, t_0)} u^2 dx dt \leq C \ve^{2/n} \int_{Q_{2r}(x_0, t_0)} u^2 dx dt,
\end{equation*} 
see \eqref{e5} below. This critical information proves that $u$ vanishes to infinite order at $(x_0,t_0)$. We can thus invoke the nonlocal space-like strong unique continuation property in Theorem \ref{main1} to
finally infer that it must be $u(\cdot, t_0) \equiv 0$ in $\Rn$. This contradicts the initial assumption that $u\not\equiv 0$ in $B_1 \times (-1, 0]$, thus allowing us to spread the zero set of $u$.

The organisation of the paper is as follows. In section \ref{s:n} we introduce some notation and gather the above mentioned results that are used in the proof of Theorem \ref{main} in section \ref{s:m}. In closing, we mention that for the existing literature on strong unique continuation for $(-\Delta)^s$ and its parabolic counterpart $(\partial_t - \Delta)^s$, the reader should see  \cite{FF, Ru1, BG, FPS, BGh, AT}.

%%%%%%%%%%%%%%%%%%%%%%%%%%%%%%%%%%%%%%%%%%% 
 
\section{Notations and Preliminaries}\label{s:n}

In this section we introduce the relevant notation and gather some auxiliary results that will be useful in the rest of the paper. Generic points in $\Rn \times \R$ will be denoted by $(x_0, t_0), (x,t)$, etc. For an open set $\Omega\subset \Rn_x\times \R_t$, we indicate with $C_0^{\infty}(\Omega)$ the set of compactly supported, smooth functions in $\Om$. We also denote by $H^{\alpha}(\Omega)$ the non-isotropic parabolic H\"older space, see \cite[p. 46]{Li}. 
Given the nonlocal operator \eqref{sHft}, the parabolic Sobolev space of fractional order $2s$ is
\begin{align}\label{dom}
\mathscr H^{2s} & =  \operatorname{Dom}(H^s)   = \{f\in \mathscr S'(\R^{n+1})\mid f, H^s f \in L^2(\R^{n+1})\}
\\
&  = \{f\in L^2(\R^{n+1})\mid (\xi,\sigma) \to (4\pi^2 |\xi|^2 + 2\pi i \sigma)^s  \hat f(\xi,\sigma)\in L^2(\R^{n+1})\}.
\notag
\end{align} 

Hereafter in this paper, we assume that $u\in \mathscr H^{2s}$ solves the equation \eqref{e0}, where on the potential $V$ we make the hypothesis that for some $K>0$ one has 
\begin{equation}\label{vassump}
||V||_{C^1(B_1 \times (-1, 0])} \leq K, \ \text{if}\ s \in [1/2, 1),\ \ \ \ \ 
||V||_{C^2(B_1 \times (-1, 0])} \leq K,\ \text{for}\ s \in (0,1/2).
\end{equation}

Under such assumptions on $u$ and $V$, various basic results hold. In order to state them we next recall the counterpart  for the parabolic nonlocal operator $H^s$ of the Caffarelli-Silvestre extension problem in \cite{CS}. When $s=1/2$ such problem was originally introduced and solved by F. Jones in \cite{Jr}, and later independently and more extensively developed for any $s\in (0,1)$ by N\"ystrom and Sande \cite{NS} and Stinga and Torrea in \cite{ST}. Since we need to consider the half-space $\R^{n+1}_{(x,t)} \times \R^+_z$, it will be convenient to combine the ``extension" variable $z>0$ with $x\in \Rn$, and indicate the generic point in the thick space $\Rn_x\times\R_z$ with the letter $X=(x,z)$. Whenever convenient, we will indicate with the short notation $U(X,t)$ the value at the point $((x,t),z)$ of a function $U:\R^{n+1}_{(x,t)} \times \R^+_z\to \R$. This should not cause any confusion in the reader's mind. For $x_0\in \Rn$ and $r>0$ we let $B_r(x_0) = \{x\in \Rn\mid |x-x_0|<r\}$, and denote the upper half-ball by
$\mathbb B_r^+(x_0,0)=\{X = (x,z) \in \R^n \times \R^{+}\mid |x-x_0|^2 + z^2 < r^2\}$. The parabolic cylinder in the thin space $Q_r(x_0, t_0) = B_r(x_0) \times [t_0, t_0 + r^2)$. We also will need the upper half-cylinder in thick space $\mathbb Q_r^+((x_0,t_0),0)=\mathbb B_r^+(x_0,0) \times (t_0,t_0+r^2]$.  
When the center $x_0$ of $B_r(x_0)$ is not explicitly indicated, then we are taking $x_0 = 0$. Similar agreement for the thick half-balls $\mathbb B_r^+(x_0,0)$. 

\begin{rmrk}\label{R:balls}
Since in this paper we make extensive use of the work \cite{ABDG}, it is important that we alert the reader that what we are presently indicating with $\mathbb B_r^+(x_0,0)$ and $\mathbb Q_r^+((x_0,t_0),0)$ were respectively denoted by $\mathbb B_r(x_0,0)$ and $\mathbb Q_r((x_0,t_0),0)$ on p. 6 in \cite{ABDG}. The reader should keep this in mind when comparing the definition of the quantity $\theta$ in \eqref{theta} with the one in (3.3) in that paper.  
\end{rmrk}

For notational ease, $\nabla U$ and  $\operatorname{div} U$ will respectively refer to the operators  $\nabla_X U$ and $ \operatorname{div}_X U$.  The partial derivative in $t$ will be denoted by $\p_t U$, and also at times  by $U_t$. The partial derivative $\partial_{x_i} U$  will be denoted by $U_i$. At times,  the partial derivative $\partial_{z} U$  will be denoted by $U_{n+1}$.

  Given a number $a\in (-1,1)$ and a $u:\R^n_x\times \R_t\to \R$, we seek a function $U:\R^n_x\times\R_t\times \R_z^+\to \R$ that satisfies the Dirichlet problem
\begin{equation}\label{la}
\begin{cases}
\La U \overset{def}{=} \partial_t (z^a U) - \operatorname{div}(z^a \nabla U) = 0,
\\
U((x,t),0) = u(x,t),\ \ \ \ \ \ \ \ \ \ \ (x,t)\in \R^{n+1}.
\end{cases}
\end{equation}
Denote by $\py$ the weighted normal derivative
\begin{equation}\label{nder}
\py U((x,t),0)\overset{def}{=}   \underset{z \to 0^+}{\lim}  z^a \partial_z U((x,t),z).
\end{equation}
Then, the most basic property of the Dirichlet problem \eqref{la} is that if $s = \frac{1-a}2\in (0,1)$, then one has in  $L^2(\R^{n+1})$
\begin{equation}\label{np}
2^{-a}\frac{\Gamma(\frac{1-a}{2})}{\Gamma(\frac{1+a}{2})} \py U((x,t),0)=  - H^s u(x,t),
\end{equation}
where $\Gamma(x) = \int_0^\infty t^{x-1} e^{-t} dt$ is Euler gamma function. In \cite[Corollary 4.6]{BG} we proved that if $u\in \mathscr H^{2s}$, then the function $U$ in \eqref{la} is a weak solution of \begin{equation}\label{wk}
\begin{cases}
\La U=0 \ \ \ \ \ \ \ \ \ \ \ \ \ \ \ \ \ \ \ \ \ \ \ \ \ \ \ \ \ \ \text{in}\ \R^{n+1}\times \R^+_z,
\\
U((x,t),0)= u(x,t)\ \ \ \ \ \ \ \ \ \ \ \ \ \ \ \ \text{for}\ (x,t)\in \R^{n+1},
\\
\py U((x,t),0)=  2^{a} \frac{\Gamma(\frac{1+a}{2})}{\Gamma(\frac{1-a}{2})} V(x,t) u(x,t)\ \ \ \ \text{for}\ (x,t)\in B_1 \times (-1,0].
\end{cases}
\end{equation}
 Note that the third equation in \eqref{wk} is justified by \eqref{e0} and \eqref{np}. 
For notational purposes, it will be convenient henceforth to work with the following backward version of problem \eqref{wk}
\begin{equation}\label{exprob}
\begin{cases}
z^a \partial_t U + \operatorname{div}(z^a \nabla U)=0\ \ \ \ \ \ \ \ \ \ \ \ \text{in} \ \R^{n+1} \times \R^+_z,
\\	
U((x,t),0)= u(x,t)
\\
\py U((x,t),0)= V(x,t) u(x,t)\ \ \ \ \ \ \ \ \ \text{in}\ B_4 \times [0,16).
\end{cases}
\end{equation}
We note that the former can be transformed into the latter by changing $t \to -t$ and a parabolic rescaling $U_{r_0}(X,t)= U(r_0X, r_0^2t)$ for small enough $r_0$. We also emphasise that, to simplify the notation in \eqref{exprob}, we have incorporated in the potential $V$ the normalising constant $2^{a} \frac{\Gamma(\frac{1+a}{2})}{\Gamma(\frac{1-a}{2})}$ in \eqref{wk}.  

In the next section we will need the following regularity result established in \cite[Section 5]{BG}.

\begin{lemma}\label{reg1}
Let $U$  be a weak solution of \eqref{exprob}. Then there exists $\alpha>0$ such that one has up to the thin set $\{z=0\}$ 
\[
U_i,\ U_t,\ z^a U_z\ \in\ H^{\alpha}(\mathbb B_\frac{1}{2}^+ \times [0, 1/4)),\ \ \ \  i=1,2,..,n.
\]
Moreover, the relevant H\"older norms are bounded by $\int_{\mathbb B_1^+ \times (0, 1)} U^2 z^a dX dt$.  Furthermore, we also have 
\begin{equation}\label{est}
\int_{\mathbb B_{1/2}^+ \times [0, 1/4)} ( |\nabla_x U|^2 + |\nabla \nabla_x U|^2 + U_t^2 ) z^a dX dt \leq C \int_{\mathbb B_1^+ \times (0, 1)} U^2 z^a dX dt.
\end{equation}
\end{lemma}

In the proof of Theorem \ref{main} we will need the following basic conditional doubling property established in \cite[Theorem 3.5]{ABDG}. With $U$ as in \eqref{exprob},  define 
\begin{equation}\label{theta}
\theta \overset{def}{=}\frac{\int_{\mathbb Q_4^+} U(X,t)^2 z^adXdt }{\int_{\mathbb B_1^+} U(X,0)^2 z^adX}.
\end{equation}
Notice that in \eqref{theta} the notation $U(X,t)$ means $U((x,t),z)$. Similarly, $U(X,0)$ indicates $U((x,0),z)$. Concerning the definition of $\theta$, the reader should also keep Remark \ref{R:balls} in mind.
	
\begin{thrm}\label{doub}
Let $U$ be a solution of \eqref{exprob}. There exists $N>2$, depending on $n$, $a$ and the $C^1$-norm of $V$, such that $N\log(N\theta) \ge 1$, and for which:
\begin{itemize}
\item[(i)] For $r \leq 1/2,$ we have 
$$\int_{\mathbb B_{2r}^+}U^2(X,0)z^adX \leq (N \theta)^N\int_{\mathbb B_{r}^+}U^2(X,0)z^adX.$$
\end{itemize}
Moreover, for  $r \leq 1/\sqrt{N \operatorname{log}(N \theta)}$  the following inequality holds:
\begin{itemize}
\item[(ii)]$$\int_{\mathbb Q_{2r}^+} U^2 z^adXdt \leq \operatorname{exp}(N \operatorname{log}(N \theta) \operatorname{log}(N \operatorname{log}(N \theta)))\int_{\mathbb Q_r^+}U^2z^adXdt.$$ 
\end{itemize}
\end{thrm}	
	
We say that a function $u(x,t)$ vanishes to infinite order at $(0,0)$ if for all $k>0$ one has for $r \to 0$
\begin{equation}\label{vp}
\int_{B_r \times (-r^2, 0]} u(x,t)^2 dx dt  = O(r^k).
\end{equation}

\begin{rmrk}\label{R:norms}
It is worth mentioning here that in the following strong unique continuation results from \cite{ABDG}, Theorems \ref{extmain} and \ref{main1}, which will be needed in the proof of Theorem \ref{main}, the notion of vanishing to infinite order used the $L^\infty$ norm, instead of the $L^2$ norm as in \eqref{vp}. Since the proofs hold unchanged with $L^2$ norms, and since such norms will play an important role in the blowup analysis, we will use \eqref{vp}. 
\end{rmrk}

Theorem \ref{doub}, combined with an appropriate blow-up analysis, was used in \cite{ABDG} to derive the following strong unique continuation property which represents the central result of that work.  

\begin{thrm}\label{extmain}
Let $U$ be a solution of \eqref{exprob}  in $\mathbb Q_{2}^+$ with $V$ satisfying the assumption in \eqref{vassump}. If  the function $u:Q_2 \to \mathbb R$ defined by $u(x,t) \overset{def}= U((x,t), 0)$ vanishes to infinite order at $(0,0)$, then $U((x,0),z) \equiv 0$ in $\mathbb B_{2}^+$.
\end{thrm}

\begin{proof}
Theorem \ref{extmain} was not explicitly stated in \cite{ABDG}, but its proof is embedded in that of \cite[Theor. 1.1, p. 35-38]{ABDG}. We thus refer the reader to that source.

\end{proof}

The next result is the just quoted Theorem 1.1 from \cite{ABDG}. Since this result will be used in the proof of Theorem \ref{main}, we need to make a comment here.    

\begin{thrm}\label{main1}
Let $u \in \operatorname{Dom}(H^s)$ solve \eqref{e0} in $B_1 \times (-1, 0]$, and assume that $u$ vanishes  to infinite order at $(0,0)$.
Then it must be $u(\cdot, 0) \equiv 0$ in $\Rn$.
\end{thrm}

In Section \ref{s:m} we will also need the following trace interpolation estimate in \cite[Lemma 2.4]{ArB1}, inspired to a related time-independent result in \cite{RS}. 

\begin{lemma}\label{interpolation}
Let $s \in (0,1)$. There exists a constant $C(n,s)>0$ such that, for any $0<\eta <1$ and $f \in C^2_0(\R^n \times \R_+)$, the following holds
\begin{align}\label{inte}
\int_{\Rn} |\nabla_x f|^2 dx \le C \eta^{2s} \int_{\R^n \times \R_+} ( |\nabla_x f|^2 + |\nabla \nabla_x f|^2) z^a dX dt + C\eta^{-2} ||f||^2_{L^2(\R^n)}.
\end{align}
\end{lemma}
	
Using Lemmas \ref{reg1} and \ref{interpolation}, by arguing as in \cite[(5.14)-(5.16) on p.195]{ArB1}, we obtain the following estimate.

\begin{lemma}\label{L:inter1}
Let $U$ be as in \eqref{exprob}. For any $0< \eta <1$,  there exists $C>0$, depending on $n, a, K$, such that the following holds
\begin{equation*}
\int_{Q_1} |\nabla_x u|^2 dxdt \leq C\eta^{2s} \int_{\mathbb Q_4^+} U^2 z^a dXdt + C\eta^{-2} \int_{Q_2} u^2 dxdt.
\end{equation*}
\end{lemma}

Finally, we also need the following well-known form of Poincar\'e inequality, see e.g. \cite[Lemma 3.4, p. 54]{LU}.

\begin{lemma}\label{po}
Let $v \in W^{1,2}(B_1)$ and denote by $F= \{x\in B_1\mid v(x)=0\}$. If $E\subset B_1$ is any measurable set, one has for some $C=C(n)>0$, 
\begin{equation}\label{po1}
\left(\int_{E} v^2\right)^{1/2} \leq  C\ \frac{|E|^{1/n}}{|F|} \left(\int_{B_1} |\nabla_x v|^2\right)^{1/2}.
\end{equation}

\end{lemma}

\vskip 0.3in

\section{Proof of Theorem \ref{main}}\label{s:m}

We begin by proving a quantitative result that allows to concentrate near the thin space the $L^2$ norm in the thick space of a solution to \eqref{exprob}. In the statement of the next lemma we denote by $\tilde V(x,t)$ a function satisfying the hypothesis \eqref{vassump}.

\begin{lemma}\label{L:pos}
Let $W$ be a solution to 
\begin{equation}\label{ex1}
\begin{cases}
z^a \partial_t W + \operatorname{div}(z^a \nabla W)=0\ \ \ \ \ \ \ \ \ \ \ \ \text{in} \ \mathbb Q_4^+
\\	
\py W((x,t),0)= \tilde V(x,t) W((x,t),0)\ \ \ \ \ \ \ \ \ \text{in}\ B_4 \times [0,16),
\end{cases}
\end{equation}
such that for some $C_1>1$ one has 
\begin{equation}\label{ret}
\int_{\mathbb Q_{2}^+} W^2 z^a dXdt  \leq C_1.
\end{equation}
Assume furthermore that 
\begin{equation}\label{norm}
\int_{\mathbb Q_1^+} W^2 z^a dXdt =1.
\end{equation}
Then there exists $C_0>0$, depending on $n, a, K$ and  $C_1$, such that 
\begin{equation}\label{pos1}
\int_{Q_1} W((x,t),0)^2 dxdt \geq  C_0.
\end{equation}
\end{lemma}
\begin{proof}
We argue by contradiction, and assume that for every $k\in \mathbb N$ there exists $\tilde V_k$ satisfying \eqref{vassump}, and a solution  to \eqref{ex1}, $W_k$, which verifies \eqref{ret} and \eqref{norm}, and such that 
\begin{equation}\label{sm}
\int_{Q_1} W_k((x,t),0)^2dxdt \leq \frac{1}{k}.
\end{equation}
By the theorem of Ascoli-Arzel\`a, we can extract a subsequence $\tilde V_k \to \tilde V_\infty$, uniformly in $\overline{\mathbb Q_1^+}$. Clearly, the limit function $\tilde V_\infty$ will satisfy \eqref{vassump}. Because of \eqref{ret} and the regularity estimate in Lemma \ref{reg1}, it follows that, up to a subsequence, $W_k \to W_\infty$. Furthermore, $W_\infty$ is a weak solution  to
\begin{equation}\label{Winfty}
\begin{cases}
z^a \partial_t W_{\infty} + \operatorname{div}(z^a \nabla W_{\infty})=0\ \ \ \ \ \ \ \ \ \ \ \ \ \ \ \ \ \ \ \ \ \ \ \text{in} \ \mathbb Q_2^+,
\\	
\py W_\infty((x,t),0)= \tilde V_\infty (x,t) W_{\infty} ((x,t),0)\ \ \ \ \ \ \ \ \ \text{in}\ B_2 \times [0,4).
\end{cases}
\end{equation}
Because of uniform convergence and \eqref{sm}, it follows that $w(x,t) = W_{\infty}((x,t),0)=0$ for all $(x,t)$ in $Q_1$. This implies, in particular, that $w(x,t)$ vanishes to infinite order at \emph{every point} $(x_1,t_1)\in Q_1$. Invoking Theorem \ref{extmain} we thus infer that $W_\infty((x,0),z) \equiv 0$ in $\mathbb B_{1}^+$. Consider now any $t_1\in [0,1)$ and denote $\tilde W_\infty((x,t),z) = W_\infty((x,t+t_1),z)$, and $\tilde w(x,t) = \tilde W_\infty((x,t),0)$. From what we have observed above, $\tilde w(x,t)$ vanishes to infinite order at $(0,0)$. Since the extension equation in \eqref{ex1} is translation-invariant in $t$, again by Theorem \ref{extmain} we conclude that $W_\infty((x,t_1),z) = \tilde W_\infty((x,0),z) \equiv 0$ in $\mathbb B_{1}^+$. By the arbitrariness of $t_1\in (0,1)$, we finally conclude that $W_\infty \equiv 0$ in $\mathbb Q_1^+$. On the other hand, again using uniform convergence and \eqref{norm}, we can also assert that the following holds
\[
\int_{\mathbb Q_1^+} W_{\infty}^2 z^a dXdt =1,
\]
thus reaching a contradiction. This proves the lemma.

\end{proof}

We now turn to the 

\begin{proof}[Proof of Theorem \ref{main}]

In keeping with the presentation in Section \ref{s:n}, we will work with the adjoint nonlocal equation $(\p_t + \Delta)^s u + V(x,t) u = 0$, and the corresponding extension problem \eqref{exprob}. Without loss of generality, we assume that $F \subset B_1$ and that the adjoint version of \eqref{e0} holds in $B_4 \times [0,16)$. We will show that $u \equiv 0$ in $B_1 \times [0,16)$. Using Theorem \ref{main1}, we can then spread the zero set and reach the desired conclusion. 

We argue by contradiction and suppose that there exists $t_0 \in [0,16)$ such that 
\begin{equation}\label{assume}
u(\cdot,t_0) \not \equiv 0,\ \ \ \  \text{in}\ B_1.
\end{equation}
Since for the corresponding extension function $U$ in \eqref{exprob} we have $U((x,t_0),0) = u(x,t_0)$, by continuity we must also have $U((x,t_0),z)) \not \equiv 0$ for $(x,z)\in \mathbb B^+_1$, and therefore in particular
 \[
 \int_{\mathbb B^+_1} U((x,t_0),z))^2 z^a dX > 0.
 \]
We claim that this implies that for every $x_0\in B_1$ one has  
\begin{equation}\label{pos}
\int_{\mathbb B_1^+(x_0, 0)} U((x, t_0),z)^2 z^a dX>0.
\end{equation}
To see \eqref{pos}, consider the function $\tilde U((x,t),z) = U((x,t+t_0),z)$, which also solves \eqref{exprob}. For such function we have 
\[
\int_{\mathbb B_1^+} \tilde U(X,0)^2 z^adX = \int_{\mathbb B^+_1} U((x,t_0),z))^2 z^a dX > 0,
\]
and therefore the corresponding $\theta$ in \eqref{theta} is well-defined. We can thus apply the doubling condition (i) in Theorem \ref{doub}, obtaining for any $r>0$ (sufficiently small)
\begin{equation}\label{i}
\int_{\mathbb B_r^+} \tilde U(X,0)^2 z^a dX > 0.
\end{equation}
Given now any $x_0\in B_1$, the triangle inequality gives $\mathbb B_{1-|x_0|}^+ \subset \mathbb B^+_1(x_0,0)$. Applying \eqref{i} with $r = 1-|x_0|$, we conclude 
\[
\int_{\mathbb B_1^+(x_0, 0)} U((x, t_0),z)^2 z^a dX \ge 
\int_{\mathbb B_r^+} U((x, t_0),z)^2 z^a dX = \int_{\mathbb B_r^+} \tilde U(X,0)^2 z^a dX >0,
\]
which proves \eqref{pos}. Consider now the function $\bar{U}((x,t),z)= U((x+x_0,t+t_0),z)$, which is also a solution of \eqref{exprob}. Since by \eqref{pos} we have
\[
\int_{\mathbb B_1^+} \bar{U}(X,0)^2 z^adX = \int_{\mathbb B_1^+(x_0, 0)} U((x, t_0),z)^2 z^a dX > 0,
\]
the corresponding \eqref{theta} for $\bar U$ is also well-defined, and from the doubling inequality (ii) in Theorem \ref{doub} we can assert that for some $C_1>1, r_1>0$, the following holds for all $r \leq r_1$
\begin{equation}\label{doub1}
\int_{\mathbb Q_{2r}^+((x_0, 0, t_0))} U^2 z^a dX dt \leq C_1 \int_{\mathbb Q_r^+ ((x_0, 0, t_0))} U^2 z^a dX dt.
\end{equation}
Let  now $x_0 \in B_1$ be a Lebesgue point of $F$, i.e.
\begin{equation*}
\lim_{r \to 0} \frac{|F \cap B_r(x_0)|}{|B_r(x_0)|}=1.
\end{equation*}
This implies that given any $\ve>0$, there exists $r_\ve>0$ (which without loss of generality we can assume $< \frac{r_1}{100}$) such that for all $r< r_\ve$, one has $|F \cap B_r(x_0)|> (1-\ve)|B_r(x_0)|$. This implies for $r<r_\ve$
\begin{equation}\label{l1}
|B_r(x_0)\setminus F| < \ve |B_r(x_0)|.
\end{equation}
From \eqref{l1} and the rescaled version of the Poincar\'e inequality in Lemma \ref{po} it follows for all $r < r_\ve$
\begin{equation}\label{e1}
\int_{Q_r(x_0, t_0)} u^2 dx dt\leq Cr^2 \ve^{2/n} \int_{Q_r(x_0, t_0)}|\nabla_x u|^2 dx dt.
\end{equation}
Since a change of scale with $r<1$ for the potential $V(x,t)$ decreases the bound $K$ on its $C^{k}$-norm ($k=1$ or $2$) in \eqref{vassump}, from the rescaled version of the interpolation estimate in Lemma \ref{L:inter1} it follows
\begin{equation}\label{res1}
\int_{Q_r(x_0, t_0)}|\nabla_x u|^2 dx dt \leq \frac{C\eta^{2s}}{r^{3+a}}\int_{\mathbb Q_{4r}^+((x_0, 0, t_0))} U^2 z^a dX dt + \frac{C \eta^{-2}}{r^2} \int_{Q_{2r}(x_0, t_0)} u^2 dx dt. 
\end{equation}
Using \eqref{res1} in \eqref{e1}, we obtain for a new $C>0$
\begin{align}\label{e2}
\int_{Q_r(x_0, t_0)} u^2 dx dt & \leq  C \eta^{-2} \ve^{2/n} \int_{Q_{2r}(x_0, t_0)} u^2 dx dt+ \frac{C\eta^{2s}\ve^{2/n}}{r^{1+a}} \int_{\mathbb Q_{4r}^+((x_0, 0, t_0))} U^2 z^a dX dt\\ \notag
&  \leq  C \eta^{-2} \ve^{2/n} \int_{Q_{2r}(x_0, t_0)} u^2 dxdt+ \frac{C \eta^{2s} \ve^{2/n}}{r^{1+a}} \int_{\mathbb Q_{r}^+((x_0, 0, t_0))} U^2 z^a dX dt, 
\notag
\end{align}
where in the second inequality we have used \eqref{doub1}  (which we can, since $r < r_{\ve}< \frac{ r_1}{100}$). Consider now the function 
\[
W((x,t),z)=  \frac{U((x_0 + rx, t_0 + r^2 t),rz)}{\bigg(\frac{1}{r^{n+3+a}}\int_{\mathbb Q_{r}^+((x_0, 0, t_0))} U^2 z^a dX dt\bigg)^{1/2}}. 
\]
Then $W$ solves \eqref{ex1} with a potential given by 
\[
\tilde V(x,t) = r^{2s} V(x_0 + rx, t_0 + r^2 t).
\]
Furthermore, by \eqref{doub1} and a change of variable it is seen that $W$ satisfies the bounds
\begin{equation}\label{norm1}
\begin{cases}
\int_{\mathbb Q_1^+} W^2 z^a dX dt =1,
\\
\int_{\mathbb Q_2^+} W^2 z^a dX dt \leq C_1.
\end{cases}
\end{equation}
Applying Lemma \ref{L:pos}, we reach the conclusion that $W$ satisfies the inequality \eqref{pos1}. By the change of variable $y= x_0+rx$, $\tau = t_0 + r^2 t$, we infer that $U$ verifies the estimate
\begin{equation}\label{pos12}
\int_{\mathbb Q_r^+((x_0, 0, t_0))} U^2 z^a dX dt \leq C_0^{-1} r^{1+a} \int_{Q_r(x_0, t_0)} u^2 dx dt.
\end{equation} 
We now insert \eqref{pos12} in \eqref{e2}, obtaining for yet another $C>0$ the following basic inequality
\begin{equation}\label{e4}
 \int_{Q_r(x_0, t_0)} u^2 dx dt \leq C \eta^{-2} \ve^{2/n} \int_{Q_{2r}(x_0, t_0)} u^2 dx dt + C \eta^{2s} \ve^{2/n} \int_{Q_r(x_0, t_0)} u^2 dx dt. 
\end{equation}
  Since $\ve<1$, by taking $\eta$ sufficiently small (it suffices to choose $C \eta^{2s}\le 1/2$), we can absorb the second term in the right-hand side of \eqref{e4} in the left-hand side, and finally arrive at the following bound 
\begin{equation}\label{e5}
\int_{Q_r(x_0, t_0)} u^2 dx dt \leq C \ve^{2/n} \int_{Q_{2r}(x_0, t_0)} u^2 dx dt,
\end{equation}
valid for some $C>1$. Finally, we want to show that \eqref{e5} implies that $u$ vanishes to infinite order at $(x_0,t_0)$. We first observe that, if $f:[0,1]\to [0,\infty)$ is an increasing function such that for every $\delta\in (0,1)$, there exists $r_\delta\in (0,1)$ such that $f(r)\le \delta f(2r)$ for $0<r<r_\delta$, then we have for every $0<r<r_\delta$
\begin{equation}\label{f}
f(r) \le \delta \left(\frac r{r_\delta}\right)^{\frac{\log \delta}{\log 2}} f(1).
\end{equation}
To prove \eqref{f}, we fix $m\in \mathbb N\cup\{0\}$ such that $2^{m} < \frac{r_\delta}{r} \le 2^{m+1}$. We thus find 
\[
f(1) \ge f(r_\delta) \ge f(2^{m} r) \ge \delta^{-1} f(2^{m} r) \ge\ ...\ge \delta^{-m} f(r).
\]
This gives for every $0<r<r_\delta$
\begin{equation}\label{fast}
f(r) \le f(1) \left(\frac{r}{r_\delta}\right)^{\log\left(\frac{1}{\delta}\right)^{\frac{1}{\log 2}}}.
\end{equation}
If now $k\in \mathbb N$ is arbitrarily chosen, pick $\delta(k)\in (0,1)$ such that 
\[
\log\left(\frac{1}{\delta}\right)^{\frac{1}{\log 2}} \cong k.
\]
Corresponding to such choice, in view of \eqref{fast} there exists $r_k = r_{\delta(k)}\in (0,1)$ such that for every $0<r<r_k$
\[
f(r) \cong f(1) \left(\frac{r}{r_\delta}\right)^{k}.
\]
This shows that $f(r) = O(r^k)$ as $r\to 0^+$. By the arbitrariness of $k$ we infer that $f$ vanishes to infinite order at $r= 0$. Applying these observations to 
\[
f(r) = \int_{Q_r(x_0, t_0)} u^2 dx dt,
\]
by taking $\delta = C\ve^{2/n}$ in \eqref{e5}, we reach the conclusion that $u$ vanishes to infinite order at $(x_0,t_0)$. Invoking Theorem \ref{main1}
we finally infer that it must be $u(\cdot, t_0) \equiv 0$ in $\Rn$, which obviously contradicts \eqref{assume}. We have thus proved that it must be $u \equiv 0$ in $B_1 \times [0,16)$. As we have mentioned, now we can spread the zeros of $u$ and show that $u \equiv 0$ in $\Rn \times [0,16)$.

 \end{proof}

 %%%%%%%%%%%%%%%%%%%%%%%%%%%%%%%%%%%%%%%%%%%%%%%%%%


\begin{thebibliography}{99}




\bibitem{ArB1}
V. Arya \& A. Banerjee, \emph{Quantitative uniqueness for fractional heat type operators},  Calc. Var. Partial Differential Equations \textbf{62}~ (2023), no. 7, Paper No. 195, 47 pp.	

\bibitem{ABDG}
V. Arya, A. Banerjee, D. Danielli \& N. Garofalo, \emph{Space-like strong unique continuation for some fractional parabolic equations}, J. Funct. Anal. \textbf{284} (2023), no. 1, Paper No. 109723.
\bibitem{AT}
A. Audrito \& S. Terracini, \emph{On the nodal set of solutions to a class of nonlocal parabolic reaction-diffusion equations},  arXiv:1807.10135, to appear in Memoirs of AMS.
	
	



\bibitem{BG}
A. Banerjee \& N. Garofalo, \emph{Monotonicity of generalized frequencies and the strong unique continuation property for fractional parabolic equations}, Adv. Math. 336 (2018), 149-241.

\bibitem{BGh}
A. Banerjee \& A. Ghosh, \emph{Sharp asymptotic of solutions to some nonlocal parabolic equations}, arXiv:2306.00341, to appear in DCDS Series A. 
%\bibitem{BS} A. Biswas \& P. R. Stinga, \emph{ Regularity estimates for nonlocal space-time Master equations in bounded domains}, J. Evol. Equ. 21 (2021), 503--565.

\bibitem{BZ}
N. Burq \& C. Zuily, \emph{A note on the uniqueness from sets of positive measure for time-dependent parabolic operators}, ArXiv:2010.16161

\bibitem{CS}
L. Caffarelli \& L. Silvestre, \emph{An extension problem related to the fractional Laplacian}, Comm. Partial Differential Equations 32~(2007), no. 7-9, 1245-1260.
	
	
\bibitem{DG}	
D. G. de Figueiredo \& J.  Gossez,
\emph{Strict monotonicity of eigenvalues and unique continuation},
Comm. Partial Differential Equations \textbf{17}~ (1992), no. 1-2, 339-346.
	

\bibitem{EF}
L. Escauriaza \& F. Fernandez, \emph{Unique continuation for parabolic operators. (English summary)}, Ark. Mat. \textbf{41}~ (2003), no. 1, 35-60.

\bibitem{EFV}
L. Escauriaza, F. Fernandez \& S. Vessella, \emph{Doubling properties of caloric functions}, Appl. Anal. \textbf{85}~ (2006), no. 1-3, 205-223.

\bibitem{EMZ}
L. Escauriaza, S. Montaner \& C. Zhang, 
\emph{Analyticity of solutions to parabolic evolutions and applications},
SIAM J. Math. Anal. \textbf{49}~ (2017), no. 5, 4064-4092.

\bibitem{FF}
M. Fall \& V. Felli, \emph{Unique continuation property and local asymptotics of solutions to fractional elliptic equations}, Comm. Partial Differential Equations, \textbf{39}~(2014), 354-397.


\bibitem{FPS}
V. Felli, A. Primo \& G. Siclari, \emph{On fractional parabolic equations with hardy-type potentials}, Commun. Contemp. Math. (2023), Article number 2350062, 1-50. DOI:
10.1142/S0219199723500621.




\bibitem{Jr}
B. F. Jones, \emph{Lipschitz spaces and the heat equation}. J. Math. Mech. \textbf{18}~(1968/69), 379-409. 


	
\bibitem{LU}
O. A. Ladyzhenskaya \& N. N. Ural'tseva, \emph{Linear and Quasilinear Elliptic Equations}.
Academic Press, 1968.
	
	

	
	
	



\bibitem{NS}
K. Nystr\"om \& O. Sande, \emph{Extension properties and boundary estimates for a fractional heat operator}, Nonlinear Analysis, 140~(2016), 29-37.

\bibitem{Li}
G. Lieberman, \emph{Second order parabolic differential equations},  World Scientific Publishing Co., Inc., River Edge, NJ, 1996. xii+439 pp. ISBN: 981-02-2883-X.


 

%\bibitem{Ru}
%A. R\"uland, \emph{On some rigidity properties in PDEs}, Dissertation, Rheinischen Friedrich-Wilhelms-Universit\"at Bonn, 2013.

\bibitem{Ru1}
A. R\"uland, \emph{Unique continuation for fractional Schr\"odinger equations with rough potentials}, Comm. Partial Differential Equations \textbf{40}~ (2015), no. 1, 77-114.

\bibitem{Ru2}
A. R\"uland, \emph{Unique continuation for sublinear elliptic equations based on Carleman estimates},
J. Differential Equations \textbf{265}~ (2018), no. 11, 6009-6035.
\bibitem{RS} A. R\"uland \& M. Salo, \emph{The fractional Calderón problem: low regularity and stability.} Nonlinear Anal. \textbf{193}~ (2020), 111529, 56 pp.	



\bibitem{ST} P.~R. Stinga \&  J.~L. Torrea, 
\emph{Regularity theory and extension problem for fractional nonlocal parabolic equations and the master equation}, SIAM J. Math. Anal. 49 (2017), 3893--3924.



\end{thebibliography}
\end{document}